\documentclass[a4paper,10pt,DIV=12]{scrartcl}

\usepackage[utf8]{inputenc}		%
\usepackage{amsmath}				%
\usepackage{mathtools}			%
\usepackage{amsthm,amssymb,wasysym}		%
\usepackage{amsfonts}			%
\usepackage[numbers,square]{natbib}	%
\usepackage{comment}			%
\usepackage{authblk}			%
\usepackage[babel]{csquotes}
\usepackage{url}

\usepackage[ngerman, english]{babel}
\newcommand{\id}{ {1\!\!\!\:1 } }

\DeclareMathOperator{\OO}{O}

\DeclareMathOperator{\diag}{diag}
\DeclareMathOperator{\sym}{sym}

\DeclareMathOperator{\dev}{dev}

\DeclareMathOperator{\Cof}{Cof}

\DeclareMathOperator{\GLp}{GL^+}

\DeclareMathOperator{\PSym}{PSym}
\DeclareMathOperator{\Sym}{Sym}

\DeclareMathOperator{\tr}{tr}

\newcommand{\R}{\mathbb{R}}

\newcommand{\N}{\mathbb{N}}

\newcommand{\C}{\mathbb{C}}

\newcommand{\Rn}{\mathbb{R}^n}
\newcommand{\Rmat}[2]{\mathbb{R}^{#1\times #2}}
\newcommand{\Rnn}{\Rmat{n}{n}}

\newcommand{\GLpn}{\GLp(n)}

\newcommand{\On}{\OO(n)}

\newcommand{\Symn}{\Sym(n)}
\newcommand{\PSymn}{\PSym(n)}

\newcommand{\ds}{\mathrm{ds}}
\newcommand{\dt}{\mathrm{dt}}

\newcommand{\dz}{\mathrm{dz}}

\newcommand{\dd}[1]{\frac{\mathrm{d}}{\mathrm{d#1}}}

\newcommand{\ddh}{\dd{h}}
\newcommand{\ddt}{\dd{t}}

\newcommand{\grad}{\nabla}

\newcommand{\khat}{\hat{k}}
\newcommand{\Htilde}{\widetilde H}

\newcommand{\sigmahat}{\widehat \sigma}

\newcommand{\tel}[1]{\frac{1}{#1}}
\newcommand{\half}{\tel{2}}

\newcommand{\norm}[1]{\Vert #1 \Vert}

\newcommand{\innerproduct}[1]{\langle #1 \rangle}

\newcommand{\eval}[2]{\left.{#1}\right|_{#2}}
\newcommand{\setvert}{\,|\,}
\newcommand{\matr}[1]{\begin{pmatrix} #1 \end{pmatrix}}
\newcommand{\matrs}[1]{\left(\begin{smallmatrix} #1 \end{smallmatrix}\right)}
	
\newcommand{\emptydmatr}[3]{\matr{#1&&\\ &#2&\\ &&#3}}	
\newcommand{\nnl}{\nonumber \\}
\newcommand{\inv}{^{-1}}
\newcommand{\eqq}{\;=\;}

\newcommand{\WeH}{W_{\textrm{eH}}}
\newcommand{\WH}{W_{\textrm{H}}}

\newcommand{\pdd}[2]{\frac{\partial #1}{\partial #2}}

\theoremstyle{plain}
\newcounter{theoremCounter}
\numberwithin{theoremCounter}{section}
\newtheorem{lemma}[theoremCounter]{Lemma}
\newtheorem{proposition}[theoremCounter]{Proposition}

\newtheorem{corollary}[theoremCounter]{Corollary}

\theoremstyle{definition}

\newtheorem{remark}[theoremCounter]{Remark}

\usepackage{mathrsfs}
\usepackage{enumitem}
\usepackage[nottoc]{tocbibind}

\newcommand{\Hspace}{V}
\newcommand{\analytic}{\mathscr{H}}
\newcommand{\intfrac}{\frac{1}{2\pi i}}
\newcommand{\cproduct}[1]{\innerproduct{#1}_\C}
\newcommand{\fulldiag}{\diag(\lambda_1,\dotsc,\lambda_n)}
\newcommand{\fullmatr}{\matrs{\lambda_1&&\\&\ddots&\\&&\lambda_n}}
\newcommand{\Hoff}{H^{\mathrm{off}}}
\newcommand{\Hdiag}{H^{\mathrm{diag}}}
\newcommand{\Adiag}{A^{\mathrm{diag}}}

\title{Some remarks on the monotonicity\\ of primary matrix functions on the set of symmetric matrices}
\author{
Robert Martin\thanks{Corresponding author, Lehrstuhl f\"{u}r Nichtlineare Analysis und Modellierung, Fakult\"{a}t f\"{u}r Mathematik, Universit\"{a}t Duisburg-Essen,  Thea-Leymann Str. 9, 45127 Essen, Germany, email: robert.martin@uni-due.de} \quad
and \quad
Patrizio Neff\thanks{Head of Lehrstuhl f\"{u}r Nichtlineare Analysis und Modellierung, Fakult\"{a}t f\"{u}r Mathematik, Universit\"{a}t Duisburg-Essen,  Thea-Leymann Str. 9, 45127 Essen, Germany, email: patrizio.neff@uni-due.de}
}
\date{\today}

\begin{document}
\maketitle

\begin{abstract}
This note contains some observations on primary matrix functions and different notions of monotonicity with relevance towards constitutive relations in nonlinear elasticity. Focussing on primary matrix functions on the set of symmetric matrices, we discuss and compare different criteria for monotonicity. The demonstrated results are particularly applicable to computations involving the \emph{true-stress-true-strain} monotonicity condition, a constitutive inequality recently introduced in an Arch.\ Appl.\ Mech.\ article by C.S. Jog and K.D. Patil. We also clarify a statement by Jog and Patil from the same article which could be misinterpreted.
\end{abstract}

\tableofcontents

\newpage
\section{Preliminaries}
This note has been inspired by our reading of Jog's and Patil's interesting work on elastic stability \cite{Jog2013} which is full of new ideas and insights, notably the inspiring introduction of the \emph{true-stress-true-strain} monotonicity condition (c.f. \cite{NeffGhibaLankeit})
\begin{align}
	&\innerproduct{\sigma(X)-\sigma(Y),\, X-Y}\geq 0\quad \forall X,Y\in\Sym(3)\,, \tag{TSTS-M$^+$}\label{eq:jogInequality}\\[2mm]
	&\sigma(\log V) = \tel{\det V}\cdot\tau(\log V) = e^{-\tr(\log V)}\cdot\underbrace{\partial_{\log V}\, W(\log V)}_{=:\tau(\log V)}\,,\nonumber
\end{align}
where $\sigma$ is the \emph{Cauchy-stress} (or \emph{true stress}) tensor considered as a function of the \emph{logarithmic} (or \emph{true}) \emph{strain} $\log V$,\: $V=\sqrt{FF^T}$ is the \emph{left Biot-stretch} tensor and \[\innerproduct{M,N}:=\tr(M^TN)=\sum_{i,j=1}^n X_{i,j}Y_{i,j}\] denotes the canonical inner product on $\Rnn$.
Inequality \eqref{eq:jogInequality}, which can also be stated as %
\begin{equation}
	\sym\,\pdd{\sigma(\log V)}{\log V} \quad \text{ is positive definite,} %
\end{equation}
was originally used by Jog and Patil \cite{Jog2013} to characterize material instabilities in elastic materials.
While inequality \eqref{eq:jogInequality} is not fulfilled by the stress response induced by the \emph{isotropic Hencky energy}\cite{Anand79,Anand86,Hencky1929,henckyTranslation,Neff_Eidel_Osterbrink_2013}
\[
	\WH = \mu\,\norm{\dev_n \log V}^2 \,+\, \frac\kappa2 \, [\tr(\log V)]^2
\]
with the shear modulus $\mu$ and the bulk modulus $\kappa$, the energy function
\begin{equation}
	W = \frac{\mu}{k} \, e^{k\,\norm{\log V}^2} \,+\, \frac{\lambda}{2\,\khat} \, e^{\khat\,[\tr(\log V)]^2}\,, \qquad k>\frac38\,,\; \khat>\frac18\,, \label{eq:TSTSenergy}
\end{equation}
which approximates the Hencky energy for sufficiently small strains, satisfies \eqref{eq:jogInequality} on all of $\PSymn$ \cite[Corollary 4.1]{NeffGhibaLankeit}; here, $\tr X = \sum_{i=1}^n X_{i,i}$ is the trace of $X\in\Rnn$, $V=\sqrt{FF^T}$ is the left Biot-stretch tensor, $\norm{X} = \sqrt{\tr X^TX}$ denotes the Frobenius matrix norm, $\dev_n X = X-\frac{\tr X}{n} \,\id$ is the deviatoric part of $X$ and $\mu,\lambda$ are the two Lam\'e constants. Furthermore, another variant recently introduced as the \emph{exponentiated Hencky energy} \cite{NeffGhibaLankeit,NeffGhibaExpHenckyPart3,NeffGhibaExpHenckyPart2}
\[
	\WeH = \frac{\mu}{k} \, e^{k\,\norm{\dev_n \log V}^2} \,+\, \frac{\kappa}{2\,\khat} \, e^{\khat\,[\tr(\log V)]^2}
\]
with dimensionless parameters $k>\frac13$ and $\khat>\frac18$ fulfils \eqref{eq:jogInequality} on the conical \enquote{elastic domain}
\[
	\mathcal{E} = \left\{ V\in\PSymn : \norm{\dev \log V}^2 \leq \frac23\, \text{\boldmath$\displaystyle\mathbf{\sigma}$}_{\mathrm{y}}^2 \right\}
\]
for a given yield stress $\text{\boldmath$\mathbf{\sigma}$}_{\mathrm{y}}$ \cite[Remark 4.1]{NeffGhibaLankeit}. For other well-known energy functions like Neo-Hooke, Mooney-Rivlin or the Ogden energy, \eqref{eq:jogInequality} is not satisfied. Until \eqref{eq:TSTSenergy} it was not even clear whether there exists an isotropic hyperelastic formulation satisfying \eqref{eq:jogInequality} at all.

We believe that the true-stress-true-strain monotonicity condition has the potential to greatly advance the subject of constitutive requirements in nonlinear elasticity.
Therefore, we find it apt to shed some light on different notions of monotonicity and their interconnections which arise in nonlinear elasticity in general as well as in computations for checking inequality \eqref{eq:jogInequality} in particular.
Since many of the stress tensors in nonlinear elasticity are symmetric, we consider in the following matrix functions mapping a convex subset of $\Symn$ to the set $\Symn$ of symmetric matrices. Of particular interest is the monotonicity of the principal matrix logarithm $\log$ on the set $\PSymn$ of positive definite matrices.

\subsection{A simple observation on monotonicity}
Let $\Hspace$ be a finite-dimensional Hilbert space with the inner product $\innerproduct{\cdot,\cdot}$ and let $M\subset \Hspace$ be a convex open subset of $\Hspace$. A function $f:M\to \Hspace$ is called \emph{monotone} (or \emph{Hilbert space monotone}) on $M$ if \[\innerproduct{f(A)-f(B),\, A-B}\geq 0\] for all $A,B\in M$, and it is called \emph{strictly monotone} if \[\innerproduct{f(A)-f(B),\, A-B} > 0\] for all $A\neq B\in M$. It is well known that a function $f\in C^1(M,\Hspace)$ is monotone on $M$ if and only if \[\innerproduct{Df[A].H,\,H}\geq 0\] for all $H\in \Hspace$, i.e.\ if and only if the Fr\'echet derivative $Df[A]$ is positive semi-definite for all $A\in M$, and it is strictly monotone on $M$ if
\begin{equation}
	\innerproduct{Df[A].H,\,H}> 0 \label{eq:strongStrictPositiveDefiniteness}
\end{equation}
for all $H\in \Hspace$, i.e.\ if $Df[A]$ is positive definite for all $A\in M$. Note, however, that \eqref{eq:strongStrictPositiveDefiniteness} is not a necessary condition for strict monotonicity.

The following lemma shows that for a continuously differentiable function $f$ on a convex set whose derivative $Df$ is self-adjoint and invertible everywhere, the positive definiteness of $Df$ in a single point is sufficient for $f$ to be strictly monotone everywhere.

\begin{lemma}
\label{lemma:posDefSinglePoint}
Let $M\subset \Hspace$ be a convex open subset of $\Hspace$, and let $f\in C^1(M,\Hspace)$ satisfy
\begin{alignat*}{3}
&\text{i)} && \quad\exists\:A_0\in M: \quad &&Df[A_0]\text{ is \emph{positive definite},}\\
&\text{ii)} && \quad\forall\:A\in M: \quad &&Df[A]\text{ is \emph{invertible} and \emph{self-adjoint}.}
\end{alignat*}
Then $\innerproduct{Df[A].H,\, H} > 0$ for all $H\in V$ and thus $f$ is strictly monotone on $M$.
\end{lemma}
\begin{proof}
Assume that $f$ is not monotone on $M$. Then there exists $A_1\in M$ such that $Df[A_1]$ is not positive semi-definite. Since $f$ is continuously differentiable, the function
\begin{equation*}
	\varphi: M\to\R:\quad \varphi(A) = \lambda_{\min}(Df[A])
\end{equation*}
mapping $A$ to the smallest eigenvalue of $Df[A]$ is continuous on $M$; note that the mapping of a matrix to its smallest eigenvalue is continuous on a set of self-adjoint tensors.\\
The set $M$ is convex (and thus connected) by assumption, hence we can choose a curve $\gamma\in C^1([0,1];M)$ with $\gamma(0)=A_0$, $\gamma(1)=A_1$ and obtain
\begin{align*}
	\varphi(\gamma(0)) &\:=\: \lambda_{\min}(Df[A_0]) \:>\: 0\,,\nnl
	\varphi(\gamma(1)) &\:=\: \lambda_{\min}(Df[A_1]) \:\leq\: 0\,.
\end{align*}
Thus there exists $a\in(0,1]$ with $\varphi(\gamma(a))=0$, according to the intermediate value theorem. But then $0$ is an eigenvalue of $Df[\gamma(a)]$ and hence $Df[\gamma(a)]$ is not invertible, contradicting ii).
\end{proof}
\begin{remark}
Note that while the proof requires $M$ to be convex in order to show monotonicity, connectedness of $M$ is sufficient to show that $Df$ is positive definite everywhere.
\end{remark}
\begin{remark}
In the one-dimensional case, Lemma \ref{lemma:posDefSinglePoint} simply states the fact that for a continuously differentiable function $f$ on $\R$ it follows from $f' \neq 0$ everywhere and $f'(a_0)>0$ for some $a_0\in\R$ that $f'>0$ everywhere on $\R$.
\end{remark}

\section{Monotonicity of primary matrix functions}
In this section we consider a \emph{primary matrix function} $f$ on the set $\Symn$ of symmetric matrices. Such a function is defined as follows\footnote{For a more general definition of primary matrix functions for non-symmetric arguments we refer to \cite[Ch. 6.2]{horn1994topics}.}:
Let $I$ be an open interval in $\R$ and let $f\in C^1(I)$. We denote\footnote{Note that $S_\R=\Symn$ and $S_{\R^+}=\PSymn$.} by $S_I$ the set of symmetric matrices with no eigenvalues outside $I$:
\[S_I := \{M\in\Symn \setvert \lambda(M)\subset I\}\,,\]
where $\lambda(M)\in\Rn$ is the ordered vector of the $n$ (not necessarily distinct) eigenvalues of $M$.
Then the primary matrix function $f:S_I\to\Symn$ is defined by
\[
	f(A) \eqq f(Q^T \fulldiag \,Q) \eqq Q^T \diag(f(\lambda_1),\dotsc,f(\lambda_n))\,Q \eqq Q^T \emptydmatr{f(\lambda_1)}{\ddots}{f(\lambda_n)}\,Q\,,
\]
where $A = Q^T \fulldiag \,Q = Q^T \fullmatr Q$, $Q\in\On$, is any orthogonal diagonalization of $A$. 
Furthermore we denote by $\innerproduct{X,Y}=\tr(X^TY)=\sum_{i,j=1}^n X_{i,j}Y_{i,j}$ the canonical inner product on $\Symn$.

\subsection[Analytic primary matrix functions on $\Symn$ and $\PSymn$]{Analytic primary matrix functions on \boldmath$\Symn$ and $\PSymn$}
For now we assume that $f\in\analytic(\R)$, where $\analytic(\R)$ is the set of analytic functions on $\R$. The more general case will be considered later on.\\
For readability reasons all lemmas, propositions and proofs will be stated for the case $f:\R\to\R$ and, correspondingly, $f:\Symn\to\Symn$. The restriction to the set of positive definite matrices (or even, for some open interval $I\subset\R$, the (convex) set $S_I$ of symmetric matrices $A$ with $\lambda(A)\subset I$) allows for nearly identical proofs.

The following lemma is stated in \cite{mathias1994} in a more general form. The proof given there is based on the expansion of $f$ into a matrix power series: observe for example that, for $f(A)=A^2$,
\[
	(A+H)^2 = A^2 + AH + HA + H^2 \quad \Rightarrow \quad Df[A].H = AH+HA
\]
and hence
\[
	\innerproduct{Df[A].H,\, \Htilde} = \innerproduct{AH+HA,\, \Htilde} = \innerproduct{AH,\, \Htilde} + \innerproduct{HA,\, \Htilde} = \innerproduct{H,\, A\Htilde+\Htilde A} = \innerproduct{H,\, Df[A].\Htilde}
\]
for $A,H,\Htilde\in\Symn$, thus $Df[A]$ is self-adjoint with respect to the canonical inner product on $\Rnn$. Similarly, the derivative of $A\mapsto A^k$ is self-adjoint for all $k\in\N$, from which one can show that the derivative of an analytic matrix function $f(A) = \sum_{k=1}^\infty \alpha_k \cdot A^k$ is self-adjoint as well.
\begin{lemma}
\label{lemma:primarySelfAdjoint}
Let $f\in \analytic(\R)$. Then the derivative of $f:\Symn\to\Symn$ is self-adjoint with respect to the canonical inner product on $\Symn$:
\begin{equation}
	\innerproduct{Df[A].H,\,\Htilde} = \innerproduct{H,\,Df[A].\Htilde} \quad \forall A,H,\Htilde\in\Symn\,.
\end{equation}
\end{lemma}
\begin{proof}
We use an integral formula given in \cite[(6.6.2)]{horn1994topics}:
\[
	DF[A].H = \intfrac \int_\Gamma f(z) (z\id-A)\inv \,H\, (z\id-A)\inv\,\dz
\]
for $A\in\Symn$, where $\gamma$ is a closed curve in $\C$ such that every eigenvalue of $A$ has winding number $1$.\\
For $H,\Htilde\in\Symn$ we compute
\begin{align}
	\innerproduct{DF[A].H,\,\Htilde} &= \innerproduct{\intfrac\int_\Gamma f(z)\,(z\id-A)\inv \,H\, (z\id-A)\inv \, \dz,\: \Htilde}\nnl
	&= \intfrac\int_\Gamma f(z)\,\cproduct{(z\id-A)\inv \,H\, (z\id-A)\inv,\: \Htilde}\,\dz\nnl
	&= \intfrac\int_\Gamma f(z)\,\cproduct{H,\: (z\id-A)^{-*} \,\Htilde\, (z\id-A)^{-*}}\,\dz\nnl
	&= \intfrac\int_\Gamma f(z)\,\cproduct{\Big((z\id-A)^{-*} \,\Htilde\, (z\id-A)^{-*}\Big)^*,\:H^*}\,\dz\nnl
	&= \intfrac\int_\Gamma f(z)\,\cproduct{\Big((z\id-A)^{-*} \,\Htilde\, (z\id-A)^{-*}\Big)^*,\:H^*}\,\dz\label{eq:integralOne}\\
	&= \intfrac\int_\Gamma f(z)\,\cproduct{(z\id-A)^{-T} \,\Htilde\, (z\id-A)^{-T},\:H}\,\dz\label{eq:integralTwo}\\
	&= \intfrac\int_\Gamma f(z)\,\cproduct{(z\id-A)\inv \,\Htilde\, (z\id-A)\inv,\:H}\,\dz\label{eq:integralThree}\\
	&= \cproduct{\intfrac\int_\Gamma f(z)\,(z\id-A)\inv \,\Htilde\, (z\id-A)\inv\,\dz,\:H} = \innerproduct{DF[A].\Htilde,\:H} = \innerproduct{H,\:DF[A].\Htilde}\,,\nonumber
\end{align}
where equality of \eqref{eq:integralOne} and \eqref{eq:integralTwo} holds due to $H$ and $\Htilde$ being real and symmetric while the symmetry of $(z\id-A)$ implies \eqref{eq:integralThree}.
\end{proof}

This lemma can now be used to obtain some interesting properties of primary matrix functions and their derivatives.

\begin{proposition}
\label{prop:specmonToHilmonAna}
Let $f\in \analytic(\R)$ with $f'(t)>0$ for all $t\in\R$. Then the primary matrix function $f:\Symn\to\Symn$ is Hilbert space monotone.
\end{proposition}
\begin{proof}
According to Lemma \ref{lemma:primarySelfAdjoint}, the derivative $Df[A]$ is self-adjoint with respect to the inner product $\innerproduct{\cdot,\cdot}$ for every $A\in\Symn$. Thus we can apply Lemma \ref{lemma:posDefSinglePoint} if we show that $DF[A]$ is invertible everywhere and positive definite in one point.\\
Let $0\neq H\in\Symn,\: H=Q^T\diag(h_1,\dotsc,h_n)Q$. Then the derivative of $f$ at $\id\in\Symn$ is
\begin{align*}
	Df[\id].H &= \lim_{t\to 0}\:\tel{t}\:[f(\id+t\,H)-f(\id)]\\
	&= \lim_{t\to 0}\:\tel{t}\:Q^T[f(\id+\diag(th_1,\dotsc,th_n))-f(1)\id]Q\\
	&= \lim_{t\to 0}\:\tel{t}\:Q^T[f(\diag(1+th_1,\dotsc,1+th_n))-\diag(f(1),\dotsc,f(1))]Q\\
	&= \lim_{t\to 0}\:\tel{t}\:Q^T[\diag(f(1+th_1)-f(1),\dotsc,f(1+th_n)-f(1))]Q\\
	&= Q^T\diag(h_1f'(1),\dotsc,h_nf'(1)) \;=\; f'(1)\,H\,,
\end{align*}
thus
\begin{align*}
	\innerproduct{Df[\id].H,\,H} = f'(1)\,\innerproduct{H,\,H}  = f'(1)\,\norm{H}^2 \:>\:0
\end{align*}
because $f'>0$ and $H\neq0$ by assumption.\\
To see that $Df[A]$ is invertible for every $A\in\Symn$ we simply note that $f$ is invertible on $\R$ and the differentiable primary matrix function $f\inv:f(\Symn)\to\Symn$ is the inverse of $f$ on $\Symn$. Then for all $A\in\Symn$, the linear mapping $Df[A]$ must be invertible as well.
\end{proof}

The next lemma shows that every analytic primary matrix function can be represented as the gradient field (differentiated with respect to $U$) of an isotropic energy function satisfying the \emph{Valanis-Landel hypothesis} \cite{valanis1967} of additive separation\footnote{The Valanis-Landel hypothesis was introduced by K.C. Valanis and R.F. Landel in 1967 as an assumption on the elastic energy potential of incompressible materials \cite{valanis1967}. Their hypothesis was later found to be in good agreement with the elastic behaviour of vulcanized rubber \cite{treloar1973elasticity}; D.F. Jones and L.R.G. Treloar concluded that \enquote{the hypothesis is valid over the range covered} in their experiments, \enquote{namely $\lambda=0.189\textendash2.62_5$} \cite{jones1975}.}:
\[
	\boxed{W(U)=\sum_{i=1}^n F(\lambda_i(U))\,,}
\]
where $\lambda_i(U)$ is the $i$-th eigenvalue of $U$. This might be considered the \enquote{hidden assumption} underlying the theory of primary matrix functions.

\begin{proposition}
\label{prop:potentialAna}
Let $f\in \analytic(\R)$. Then for any $F\in \analytic(\R)$ with $F'=f$, the function
\begin{equation}
	W: \Symn \to \R\,, \quad W(A)=\sum_{i=1}^n F(\lambda_i(A)) = \tr F(A)
\end{equation}
is a \emph{potential} of $f:\Symn\to\Symn$, i.e.
\begin{equation}
	DW[A] = f(A)
\end{equation}
or, more precisely,
\begin{equation*}
	DW[A].H = \innerproduct{f(A),\,H}
\end{equation*}
for all $A,H\in\Symn$.
\end{proposition}
\begin{proof}
Let $A,H\in\Symn$. Since $DW[A].H$ is the partial derivative of $W$ in direction $H$ at the point $A$ we find
\begin{equation*}
	DW[A].H = \lim_{t\to 0} \:\tel{t}\: [W(A+tH)-W(A)]\,.
\end{equation*}
We choose $Q\in\On$ such that $A=Q^TDQ$, where
$D=\fulldiag$,\: $\lambda_i$ denoting the eigenvalues of $A$, and compute
\begin{align*}
	DW[A].H &= \lim_{t\to 0} \:\tel{t}\: [W(A+tH)-W(A)] = \lim_{t\to 0} \:\tel{t}\:[\tr F(A+tH) - \tr F(A)]\\
	&= \lim_{t\to 0} \:\tel{t}\:\innerproduct{F(A+tH)-F(A),\,\id} = \innerproduct{\lim_{t\to 0} \:\tel{t}\:[F(A+tH)-F(A)],\,\id} = \innerproduct{DF[A].H,\,\id}\,.
\end{align*}
According to Lemma \ref{lemma:primarySelfAdjoint}, the total derivative $DF[A]$ is self-adjoint with respect to $\innerproduct{\cdot,\cdot}$ and thus
\begin{equation*}
	\innerproduct{DF[A].H,\,\id} = \innerproduct{H,\,DF[A].\id}\,.%
\end{equation*}
We find
\begin{align*}
	DF[A].\id &= \lim_{t\to 0} \:\tel{t}\: [F(A+t\,\id)-F(A)] \:=\: \lim_{t\to 0} \:\tel{t}\: [F(Q^TDQ+t\,\id)-F(Q^TDQ)]\\
	&= \lim_{t\to 0} \:\tel{t}\: [F(Q^T(D+t\,\id)Q)-Q^TF(D)Q]\\
	&= \lim_{t\to 0} \:\tel{t}\: [Q^TF(D+t\,\id)Q-Q^TF(D)Q]\\
	&= \lim_{t\to 0} \:\tel{t}\: [Q^TF(\fulldiag+t\,\id)Q-Q^TF(\fulldiag)Q]\\
	&= \lim_{t\to 0} \:\tel{t}\: Q^T[F(\diag(\lambda_1+t,\dotsc,\lambda_n+t))-F(\fulldiag)]Q\\
	&= \lim_{t\to 0} \:\tel{t}\: Q^T\Big[\diag\Big(F(\lambda_1+t)-F(\lambda_1),\dotsc,F(\lambda_n+t)-F(\lambda_n)\Big)\Big]Q\\
	&= Q^T\Big[\diag\Big(\lim_{t\to 0} \:\tel{t}\: F(\lambda_1+t)-F(\lambda_1),\dotsc,\lim_{t\to 0} \:\tel{t}\: F(\lambda_n+t)-F(\lambda_n)\Big)\Big]Q\\
	&= Q^T\Big[\diag\Big(F'(\lambda_1),\dotsc,F'(\lambda_n)\Big)\Big]Q = Q^Tf(D)\,Q \:=\: f(A)
\end{align*}
and therefore
\[
	DW[A].H = \innerproduct{DF[A].H,\,\id} = \innerproduct{H,\,DF[A].\id} = \innerproduct{H,\,f(A)} = \innerproduct{f(A),\,H}\,.\qedhere
\]
\end{proof}

\begin{remark}[Pseudo-potential]
Using the fact that $Df$ is self-adjoint everywhere, we can also obtain the potential directly by using \cite[Lemma 3.28]{NeffCISMnotes}. For $A=Q^TDQ$, $D=\diag(\lambda_1,\dotsc,\lambda_n)$, $Q\in\On$ we find
\begin{align*}
	W(A) &= \int_0^1 \innerproduct{f(t\,A),\,A}\,\dt = \int_0^1 \innerproduct{Q^Tf(t\,D)Q,\,Q^TDQ}\,\dt\\
	&= \int_0^1 \innerproduct{f(t\,D),\,D}\,\dt = \int_0^1 \tr \matr{f(t\,\lambda_1)\lambda_1&&\\&\ddots&\\&&f(t\,\lambda_n)\lambda_n}\,\dt\\
	&= \sum_{i=1}^n \int_0^1 f(\lambda_i\,t)\,\lambda_i \,\dt = \sum_{i=1}^n \int_0^{\lambda_i} f(t) \,\dt \quad=\quad \sum_{i=1}^n F(\lambda_i) + C
\end{align*}
with $C=-n\, F(0)$.
\end{remark}

\subsection{The general case}
In this section we no longer require the function $f$ to be analytic. While the results are almost identical to those of the previous subsection, the more general proofs require them to be stated in a different order.\\
The first proposition shows that every continuously differentiable primary matrix function can be represented as the gradient field (differentiated with respect to $U$) of an isotropic energy function satisfying the \emph{Valanis-Landel hypothesis}.

\begin{proposition}
\label{prop:primaryPotential}
Let $f\in C^1(I)$. Then for $F\in C^2(I)$ with $F'=f$, the function
\begin{equation}
	W: S_I \to \R\,, \quad W(A)=\sum_{i=1}^n F(\lambda_i(A)) = \tr F(A)
\end{equation}
is a \emph{potential} of $f:S_I\to\Symn$, i.e.
\begin{equation}
	DW[A] = f(A)
\end{equation}
or, more precisely,
\begin{equation*}
	DW[A].H = \innerproduct{f(A),\,H}
\end{equation*}
for all $A\in S_I,\:H\in\Symn$.
\end{proposition}
\begin{proof}
	This is a corollary to Theorem 1.1 in \cite{Lewis96derivativesof}, where it is shown that any \emph{spectral function} of the form $W(A) = g(\lambda(f))$ with a symmetric function $g\in C^1(I^n)$ is differentiable with
	\[
		DW[A] = Q^T\diag(\grad g(\lambda))Q
	\]
	for all $A\in I$ with $A=Q^T\diag(\lambda_1,\dotsc,\lambda_n)Q$. By putting $g(\lambda) = \sum_{i=1}^n F(\lambda_i)$ we find $\grad g(\lambda) = (F'(\lambda_1),\dotsc,F'(\lambda_n))$, and since $F'=f$ we obtain
	\[
		DW[A] = Q^T\diag(f(\lambda_1),\dotsc,f(\lambda_n))Q = f(A)\,.\qedhere
	\]
\end{proof}
The next lemma is due to Brown et al. \cite[Theorem 2.1]{brown2000calculus}. The proof can be found there.
\begin{lemma}
\label{lemma:primaryDifferentiability}
\pushQED{\qed}
	Let $f\in C^1(I)$. Then the primary matrix function $f:S_I\to\Symn$ is continuously differentiable on $S_I$.
\popQED
\end{lemma}
According to Proposition \ref{prop:primaryPotential} every primary matrix function $f$ on $S_I$ corresponding to $f\in C^1(I)$ has a potential. Thus the derivative of $f$ on $S_I$, which exists due to \ref{lemma:primaryDifferentiability}, is self adjoint according to Schwarz' theorem.
\begin{proposition}
\label{prop:primarySelfAdjoint}
\pushQED{\qed}
Let $f\in C^1(I)$. Then the primary matrix function $f:S_I\to\Symn$ is differentiable on $S_I$ and its derivative $Df$ is self-adjoint with respect to the canonical inner product on $\Symn$:
\begin{equation*}
	\innerproduct{Df[A].H,\,\Htilde} = \innerproduct{H,\,Df[A].\Htilde} \quad \forall A\in S_I,\;H,\Htilde\in\Symn\,.\qedhere
\end{equation*}
\popQED
\end{proposition}

Proposition \ref{prop:primaryPotential} can also be used to show how the monotonicity of $f$ on $I\subset\R$ relates to the Hilbert space monotonicity of the primary matrix function $f$ on $S_I\subset\Symn$.

\begin{proposition}
\label{prop:specmonToHilmon}
Let $f\in C^1(I)$. Then the primary matrix function $f:S_I\to\Symn$ is Hilbert space monotone if and only if $f$ is monotone on $I$.
\end{proposition}
\begin{proof}
Choose an antiderivative $F\in C^2(I)$ of $f$ and define
\[W: S_I \to \R\,, \quad W(A)=\sum_{i=1}^n F(\lambda_i(A))\,.\]
According to Proposition \ref{prop:primaryPotential}, $W$ is a potential of $f$ on $S_I$, i.e.\ $DW[A]=f(A)$ for all $A\in I$. But then $f=DW$ is monotone on $S_I$ if and only if $W$ is convex on $S_I$. According to an extension of the Chandler Davis Theorem \cite[Corollary 2]{davis1957}, this is the case if and only if the function $\lambda\mapsto\sum_{i=1}^n F(\lambda_i)$ is convex on $S_I^n$, which in turn is the case if and only if $f=F'$ is monotone on $I$.
\end{proof}
\begin{remark}
It is possible to give a proof based on the eigenvalue formula given in \cite[Theorem 2.1]{brown2000calculus}. This might be useful to distinguish monotonicity and strict monotonicity as well as positive definiteness and positive semi-definiteness of $Df$.
\end{remark}
\begin{remark}
A very similar result is also given by Norris in \cite{norris2008eulerian} and \cite[Lemma 4.1]{norris2008higherDerivatives}, where it is shown that $Df[A]$ is self-adjoint and positive definite for all $A\in \PSymn$ if the function $f:\R^+\to\R$ has the following properties:
\[
	f\in C^\infty(\R^+)\,, \quad f(1)=0\,,\quad f'(1)=1\quad \text{ and } \quad f'(t)>0 \quad \text{ for all } t\in\R^+\,.
\]
Norris calls these functions \emph{strain measures}, based on a definition by Hill given in \cite[p. 459]{hill1970} and \cite[p. 14]{hill1978}, although Norris requires the derivative $f'$ to be strictly positive, whereas Hill admits functions which are simply monotone on $\R^+$ as well.
\end{remark}

\section{Additional remarks and applications}
\subsection{The exponential function and the logarithm}
Returning to the principal logarithm $\log$ on $\PSymn$ and its inverse, the matrix exponential $\exp$ on $\Symn$, we find that Proposition \ref{prop:specmonToHilmon} immediately shows that $\log$ and $\exp$ are monotone. Furthermore, both functions are diffeomorphisms, hence their derivatives $D\log[P]$ and $D\exp[S]$ for $S\in\Symn,\:P\in\PSymn$ are invertible as well. Since the monotonicity implies that $D\log[P]$ and $D\exp[S]$ are positive semi-definite, they are therefore positive definite, thus $\log$ and $\exp$ are strictly monotone as well.\\
For these two functions we can also compute some of the aforementioned properties directly: using a representation of $D\exp$ given in \cite[Ch. 10.2]{higham2008}, we find
\begin{align*}
	\innerproduct{D\exp[A].H,\,\Htilde} &= \innerproduct{\int_0^1 \exp(sA)\,H\,\exp((1-s)A)\,\ds,\,\Htilde}\\
	&= \int_0^1 \innerproduct{\exp(sA)\,H\,\exp((1-s)A),\,\Htilde} \,\ds = \int_0^1 \innerproduct{H,\,\exp(sA)^T\,\Htilde\,\exp((1-s)A)^T} \,\ds\\
	&= \int_0^1 \innerproduct{H,\,\exp(sA)\,\Htilde\,\exp((1-s)A)} \,\ds \:=\: \innerproduct{H,\,D\exp[A].\Htilde}
\end{align*}
for $A,H,\Htilde\in\Symn$, showing that $D\exp[A]$ is self-adjoint, as well as
\begin{align*}
	\innerproduct{D\exp(A).H,\,H} &= \innerproduct{\int_0^1 \exp(sA)\,H\,\exp((1-s)A)\,\ds,\,H}\\
	&= \int_0^1 \innerproduct{H\,\exp((1-s)A),\,\exp(sA)^T\,H} \,\ds = \int_0^1 \innerproduct{\underbrace{H\,\exp((1-s)A)\,H^T}_{\text{pos. semi-definite}},\:\underbrace{\exp(sA)}_{\in\PSymn}}\,\ds \:\geq\: 0\,,
\end{align*}
showing that $D\exp[A]$ is positive semi-definite.\\
For the matrix logarithm and $A\in\PSymn,\:H,\Htilde\in\Symn$ we use, again%
\footnote{A formula for the derivative $D\log[A].H$ in a direction $H$ for commuting $A$ and $H$ as well as some properties of derivatives of primary matrix functions in arbitrary directions can be found in much earlier works by H. Richter \cite{richter1949log,richter1950}; however, Richter did not give the more general formula used here.}%
, a representation formula given in \cite[Ch. 11.2]{higham2008} to find
\begin{align*}
	\innerproduct{D\log[A].H,\,\Htilde} &= \innerproduct{\int_0^1 (t(A-\id)+\id)\inv\,H\,(t(A-\id)+\id)\inv\,\ds,\,\Htilde}\\
	&= \int_0^1 \innerproduct{(t(A-\id)+\id)\inv\,H\,(t(A-\id)+\id)\inv,\,\Htilde}\,\ds\\
	&= \int_0^1 \innerproduct{H,\,(t(A-\id)+\id)^{-T}\,\Htilde\,(t(A-\id)+\id)^{-T}} \,\ds\\
	&= \int_0^1 \innerproduct{H,\,(t(A-\id)+\id)\inv\,\Htilde\,(t(A-\id)+\id)\inv} \,\ds \;=\; \innerproduct{H,\,D\log[A].\Htilde}\,,
\end{align*}
showing that $D\log[A]$ is self-adjoint, as well as
\begin{align*}
	\innerproduct{D\log(A).H,\,H} &= \innerproduct{\int_0^1 (t(A-\id)+\id)\inv\,H\,(t(A-\id)+\id)\inv\,\ds,\,H}\\
	&= \int_0^1 \innerproduct{H\,(t(A-\id)+\id)\inv,\,(t(A-\id)+\id)^{-T}\,H} \,\ds\\
	&= \int_0^1 \innerproduct{\underbrace{H\,(t(A-\id)+\id)\inv\,H^T}_{\text{pos. semi-definite}},\:\underbrace{(t(A-\id)+\id)\inv}_{\in\PSymn}}\,\ds \;\geq\; 0\,,
\end{align*}
showing that $D\log[A]$ is positive semi-definite.\\
Note, however, that the matrix exponential is not monotone on $\R^{n\times n}$ or $\GLpn$: for $\alpha\in\R$ we compute
\begin{align*}
	&\hspace{-14mm}\innerproduct{\exp\matr{0&-\alpha\\\alpha&0} - \exp\matr{0&\alpha\\-\alpha&0},\: \matr{0&-\alpha\\\alpha&0} - \matr{0&\alpha\\-\alpha&0}}\\[2mm]
	&= \innerproduct{\matr{\cos(\alpha)&-\sin(\alpha)\\\sin(\alpha)&\cos(\alpha)} - \matr{\cos(-\alpha)&-\sin(-\alpha)\\\sin(-\alpha)&\cos(-\alpha)},\: \matr{0&-2\alpha\\2\alpha&0}}\\[2mm]
	&= \innerproduct{\matr{0&-2\sin(\alpha)\\2\sin(\alpha)&0},\: \matr{0&-2\alpha\\2\alpha&0}} \:=\: 8\,\alpha\,\sin(\alpha)
\end{align*}
and thus, for $\alpha = \frac{3\pi}{2}$,
\[
	\innerproduct{\exp\matr{0&-\alpha\\\alpha&0} - \exp\matr{0&\alpha\\-\alpha&0},\: \matr{0&-\alpha\\\alpha&0} - \matr{0&\alpha\\-\alpha&0}} = 12\,\pi\,\sin\Big(\frac{3\pi}{2}\Big) = -12\,\pi \:<\: 0\,.
\]

\subsection{Application to stress response functions in nonlinear elasticity}
We consider the \emph{Hencky constitutive model}, induced by the isotropic Hencky energy function
\[
	\WH = \mu\,\norm{\dev_n \log V}^2 \,+\, \frac\kappa2 \, [\tr(\log V)]^2\,.
\]
In this constitutive model, the \emph{Kirchhoff stress} $\tau$ corresponding to the stretch $V$ is given by
\[
	\tau = 2\,\mu\,\dev_n\log V \,+\, \kappa\tr(\log V) \,\id\,.
\]
If $2\mu=\kappa$, this relation reduces to $\tau = 2\mu\log V$, thus the mapping $V\mapsto \tau(V)$ is strictly monotone on $\PSymn$ in this special case (also called the \emph{lateral contraction free} case). However, this monotonicity does not hold for arbitrary choices of $\mu,\kappa>0$. Moreover, the mapping $\log V \mapsto \tau$ of the \emph{true strain tensor} $\log V$ to the Kirchhoff stress $\tau$ is monotone (a property also called \emph{Hill's inequality} \cite{hill1970}), while the \emph{Cauchy stress response}
\[
	V \mapsto \sigma = \frac{1}{\det V} \, \tau = \frac{1}{\det V} \, \left(2\,\mu\,\log V \,+\, 2\,\tr(\log V) \,\id \right)
\]
as well as the mapping
\[
	\log V \mapsto \sigma = \frac{1}{e^{\tr (\log V)}} \, \left(2\,\mu\,\log V \,+\, 2\,\tr(\log V) \,\id \right)
\]
are not monotone, thus the Hencky model does not satisfy the true-stress-true-strain monotonicity condition \eqref{eq:jogInequality}.

\section{Different notions of monotonicity: a comparison}
We may distinguish three types of (strict) monotonicity:
\begin{itemize}
\item The \emph{Hilbert space monotonicity}%
\begin{equation}
	\innerproduct{f(B)-f(A),\: B-A} > 0 \qquad \forall A\neq B\in\Symn \,,\label{eq:hilbertmon}\tag{H-mon}
\end{equation}
\item the \emph{operator monotonicity}%
\begin{equation}
	B-A \text{ positive definite } \quad\Rightarrow\quad f(B)-f(A) \text{ positive definite,} \label{eq:opmon}\tag{O-mon}
\end{equation}
\item the \emph{spectral monotonicity} (or monotonicity of $f$ on $\R$)%
\begin{equation}
	b > a \quad\Rightarrow\quad f(b)>f(a)\qquad \forall\: a,b\in\R\,.\label{eq:specmon}\tag{S-mon}
\end{equation}
\end{itemize}
Furthermore we consider the following condition on $f$:
\begin{equation}
	\innerproduct{f(A+H)-f(A),\: H} > 0 \qquad \forall H\in\PSymn\,,\: A\in\Symn\,.\label{eq:posdefCondition}\tag{P-mon}
\end{equation}

\begin{proposition}
\emph{Only} the following implications hold:
\begin{equation}
\eqref{eq:opmon} \Rightarrow \eqref{eq:specmon} \Leftrightarrow \eqref{eq:hilbertmon} \Leftrightarrow \eqref{eq:posdefCondition}\,.
\end{equation}
\end{proposition}

\begin{proof}~\\
$\eqref{eq:opmon} \Rightarrow \eqref{eq:posdefCondition}$: For given $H\in\PSymn$ choose $B=A+H$. Since $B-A=H$ is positive definite, \eqref{eq:opmon} implies that $f(B)-f(A)$ is positive definite as well. Thus \[\innerproduct{\underbrace{f(A+H)-f(A)}_{\in\PSymn},\: \underbrace{H}_{\in\PSymn}} > 0\,.\]
$\eqref{eq:posdefCondition} \Rightarrow \eqref{eq:specmon}$: Let $f$ satisfy condition \eqref{eq:posdefCondition}. Then, with $A=a\,\id,\,H=h\,\id,\,h\in\R^+$ we find
\begin{align*}
	0 &< \innerproduct{f(a\,\id+h\,\id)-f(a\,\id),\: h\,\id} = h\,\innerproduct{f((a+h)\,\id)-f(a\,\id),\: \id}\\
	&= h\,\tr(f((a+h)\,\id)-f(a\,\id)) = h\,\tr(f(a+h)\,\id-f(a)\,\id) \:=\; h\,n\,(f(a+h)-f(a))\,.
\end{align*}
For $a,b\in\R$ with $b>a$, choose $h=b-a$. Then $n(b-a)(f(b)-f(a))>0$ and thus $f(b)>f(a)$.\\
$\eqref{eq:specmon} \Rightarrow \eqref{eq:hilbertmon}$: Proposition \ref{prop:specmonToHilmon}\\
$\eqref{eq:hilbertmon} \Rightarrow \eqref{eq:posdefCondition}$: This implication is trivial; simply choose $B=A+H$.\\
To see that the operator monotonicity is not implied by the other conditions, consider the function $\PSymn\to\Symn,\; A\mapsto A^2$. While this function is monotone in the sense of \eqref{eq:hilbertmon} and \eqref{eq:specmon}, it is not operator monotone \cite[Example V.1.2]{bhatia1997matrix}.
\end{proof}
\begin{remark}
If $f$ is not a primary matrix function given through a scalar function on the spectrum (and \eqref{eq:specmon} is therefore not well defined), then the only generally true implications are
\begin{equation}
	\eqref{eq:opmon} \Rightarrow \eqref{eq:posdefCondition} \Leftarrow \eqref{eq:hilbertmon}\,.
\end{equation}
To see that operator monotonicity does not imply Hilbert space monotonicity in this general case, consider the function
\begin{equation}
	g: \PSymn\to\PSymn, \quad C\mapsto \det(C)\cdot\id\,. \label{eq:opMonNotHMonExample}
\end{equation}
For $A,H\in\PSymn$ we find
\[
	Dg[C].H \:=\: \ddt\eval{\det(C+tH)}{t=0} \cdot \id \:=\: \underbrace{\det(C)}_{>0\vphantom{\PSymn}}\tr(\,\underbrace{H\vphantom{()}}_{\mathclap{\in\PSymn}}\overbrace{C\inv}^{\mathclap{\in\PSymn}}) \cdot \id\,.
\]
Since $\tr(MN)>0$ for $M,N\in\PSymn$ (c.f. \cite{Neff_Diss00}) we find $\det(C)\tr(HC\inv)>0$. Thus $\det(C)\tr(HC\inv)\, \id$ is a positive definite matrix:
\[
	Dg[C].H = \det(C)\tr(HC\inv) \, \id \in \PSymn\,,
\]
hence $g$ is operator monotone. However $g$ is not Hilbert space monotone: in the case $n=2$, with $A=\matrs{3&0\\0&2}$ and $B=\matrs{5&0\\0&1}$ we find
\begin{align*}
	\innerproduct{g(B)-g(A),\,B-A} &= (\det(B)-\det(A))\,\tr(B-A) = (5-6)\,\tr\matrs{2&0\\0&-1} \:=\: -1 \:<\: 0\,.
\end{align*}
For arbitrary dimensions $n>2$ the same follows for
\[
	A=\matr{3&0&0&\dots&0\\0&2&0&\dots&\\0&0&1&\dots&\\\vdots&&&\ddots&\\0&&&&1} \qquad\text{and}\qquad B=\matr{5&0&0&\dots&0\\0&1&0&\dots&\\0&0&1&\dots&\\\vdots&&&\ddots&\\0&&&&1}\,.
\]
Note also that $Dg[C].H$ is generally not self-adjoint: for $H,\Htilde\in\Symn$ we find
\[
	\innerproduct{Dg[C].H,\,\Htilde} = \innerproduct{\innerproduct{\Cof C,\, H}\cdot\id,\: \Htilde} =  \innerproduct{\Cof C,\, H} \tr\Htilde \neq \innerproduct{\Cof C,\, \Htilde} \tr H\,.
\]
Thus \eqref{eq:opMonNotHMonExample} does not admit a potential.
\end{remark}

\section{Some observations on Jog's and Patil's calculus}
\label{section:jog}
Returning to our original motivation, namely the true-stress-true-strain inequality, we consider the equation
\[
	\pdd{\,\sigma(B)}{B} \cdot \pdd{B}{\log B} = \pdd{\,\sigmahat(\log B)}{\log B}
\]
based on the chain rule. To see how the positive definiteness of two of these terms imply the positive definiteness of the third we need the following lemma.
\begin{lemma}
\label{lemma:selfadjointProductPositiveDefinite}
Let $V$ be a finite-dimensional Hilbert space and let $A,B\in L(V,V)$ with
\begin{itemize}
	\item[i)] $A$ and $B$ are self-adjoint and positive definite,
	\item[ii)] $A B$ is self-adjoint.
\end{itemize}
Then $A B$ is positive definite.
\end{lemma}
\begin{proof}
	Since $A$, $B$ and $A B$ are self-adjoint, we find \[A B = (A B)^T = B^T A^T = B A\,,\] hence $A$ and $B$ commute. Therefore $A$ and $B$ are simultaneously diagonalizable: we can choose an orthonormal basis such that the corresponding matrices $M_A$ and $M_B$ representing $A$ and $B$ are diagonal. Since $A$ and $B$ are positive definite, all diagonal entries of $M_A$ and $M_B$ are positive. The matrix $M_{AB}$ representing $AB$ in the same basis is given by $M_{AB} = M_A \cdot M_B$ and is therefore a diagonal matrix with positive diagonal entries as well, thus $AB$ is positive definite.\\
	Note that in the case of $V=\R^n$ we can simply choose $Q\in\On$ such that $A=Q^TD_A Q$, $B=Q^TD_BQ$ with diagonal matrices $D_A, D_B$. Then $AB=Q^TD_AD_B Q$, and since the diagonal entries of $D_A$ and $D_B$ are positive, so are the diagonal entries of $D_AD_B$.
\end{proof}
We consider the derivatives
\[
	\pdd{\,\sigma(B)}{B}\,,\quad \pdd{B}{\log B}\,,\quad \pdd{\,\sigmahat(\log B)}{\log B} \quad \in \: L(\Rnn,\,\Rnn)
\]
on $\Symn$ and make the following assumptions on the functions $\sigma$ and $\sigmahat$:
\begin{alignat}{2}
	&\pdd{\,\sigma(B)}{B} \quad &&\text{ is self-adjoint,}\label{eq:assumptions1}\\
	&\pdd{\,\sigmahat(\log B)}{\log B} \quad &&\text{ is self-adjoint and positive definite in $L(\Rnn,\,\Rnn)$.}\label{eq:assumptions2}
\end{alignat}
Furthermore, we know from the previous sections that $\pdd{B}{\log B} = D\exp[\log B]$ and its inverse $\left(\pdd{B}{\log B}\right)\inv = \pdd{\log B}{B}$  $(= D\log[B])$ are self-adjoint and positive definite.\\[2mm]
Then according to Lemma \ref{lemma:selfadjointProductPositiveDefinite} the following holds:
\[
	\pdd{\,\sigma(B)}{B} \cdot \pdd{B}{\log B} = \pdd{\,\sigmahat(\log B)}{\log B} \quad \Longrightarrow \quad \pdd{\,\sigma(B)}{B} \text{\: is positive definite.}
\]
This follows directly from $\pdd{\,\sigma(B)}{B} = \pdd{\,\sigmahat(\log B)}{\log B} \cdot \left(\pdd{B}{\log B}\right)\inv$. Note that \eqref{eq:assumptions1} and \eqref{eq:assumptions2} hold if $\sigma(B)=\sigmahat(\log B)$ for a primary matrix function $\sigmahat$ induced by a monotone function $f:\R^+\to\R$ with $f'(t)>0$ for all $t \in \R^+$.

To apply Lemma \ref{lemma:selfadjointProductPositiveDefinite}, all of the involved matrices $A$, $B$ and $AB$ must be self-adjoint.
While the term \enquote{positive definite} usually implies the symmetry by definition, we will now consider matrices $A$ which are \enquote{positive definite} in the sense that $\innerproduct{Ax,x} > 0$ for all $x\in\Rn$, which is the case if and only if the symmetric part $\sym A = \half(A+A^T)$ of $A$ is positive definite.\\
We will show that the lemma does not generally hold if only one of the considered matrices is symmetric. Let
\[
	A = \matr{1&0\\0&\tel8}\,, \qquad B_t = \matr{1&-t\\0&1}
\]
for $t\in[0,1]$. Then
\[
	A\cdot B_t = \matr{1&-t \\ 0&\tel8}
\]
and we find:
\begin{itemize}
	\item $t\mapsto B_t: [0,1]\to\Rnn$ is continuous,
	\item $A$ is invertible, symmetric and positive definite,
	\item $B_t$ is invertible and \enquote{positive definite} (i.e.\ $\sym B_t = \half(B_t+B_t^T)$ is positive definite) for all $t\in[0,1]$ and
	\item $A\cdot B_t$ is invertible for all $t\in[0,1]$.
\end{itemize}
However, while $A\cdot B_0$ is obviously positive definite, the matrix $A\cdot B_1 = \matr{1&-1\\0&\tel8}$ is not since
\[
	\sym \matr{1&-1\\0&\tel8} = \matr{1&-\half\\-\half&\tel8} \quad \Longrightarrow \quad \det(\sym(A\cdot B_1)) = -\tel8 \:<\: 0\,,
\]
which implies that $\sym(A\cdot B_1)$ is not positive definite.\\
This shows not only that the lemma does not hold for non-symmetric matrices, but also that a \enquote{positive definite} (non-symmetric) matrix can be continuously deformed into a non-positive matrix without losing invertibility along the way.

In \cite[eq. (50)]{Jog2013}, Jog and Patil argue that a tensor valued function $A$ \enquote{loses positive definiteness} if and only if $B$ \enquote{loses positive definiteness}, which is deduced from the fact that $A = B C$ for some symmetric positive definite $C$. Since $A$ and $B$ are not symmetric in general, the authors define \enquote{losing positive definiteness} as the loss of invertibility. While for this definition the stated equivalence is correct, it should be carefully noted that since invertibility of a gradient is not a sufficient condition for monotonicity, this result cannot be applied to show the monotonicity of a function with gradient $A$. In particular, if $\pdd{\,\sigmahat(\log B)}{\log B}$ is not symmetric we cannot simply combine Lemma \ref{lemma:posDefSinglePoint} and Lemma \ref{lemma:selfadjointProductPositiveDefinite} to conclude that $\pdd{\,\sigma(B)}{B}$ is positive definite.

\section*{Acknowledgements}
We thank Prof. Chandrashekhar S. Jog (Indian Institute of Science) for interesting discussions on the topic of constitutive inequalities as well as Prof. Karl-Hermann Neeb (University of Erlangen) and Prof. Nicholas Higham (University of Manchester) for their helpful remarks.

{\footnotesize

\begin{appendix}
\section{Appendix}
\subsection{On the derivative of the determinant function}
Consider the first order approximation
\[
	\det(A+H) = \det A + D\det[A].H + \underbrace{\dots}_{\mathclap{\substack{\text{higher order}\\\text{terms}}}}
\]
of the determinant function at a \emph{diagonal} matrix $A=\diag(a_1,\dotsc,a_n)$. First we assume that $H\in\Symn$ is an \emph{off-diagonal} matrix of the form
\begin{equation}
	H = H^{\mathrm{off}} = \matrs{0&&h\\&\ddots&\\h&&0} \label{eq:offDiagonalMatrixSpecialCase}
\end{equation}
with $h\in\R$. We compute
\begin{align}
	&\hspace{-14mm} \det\matr{\vec{a}_1 + \matrs{0\\\vdots\\h}, \: \vec{a}_2\:\:,\:\: \cdots \:\:\,, \:\: \vec{a}_n + \matrs{h\\\vdots\\0}}\nnl
	&= \det(\vec{a}_1\,, \: \vec{a}_2\,, \: \dotsc \,, \: \vec{a}_n + \matrs{h\\\vdots\\0}) \:+\: \det(\matrs{0\\\vdots\\h}, \: \vec{a}_2\,, \: \dotsc \,, \: \vec{a}_n + \matrs{h\\\vdots\\0})\nnl
	&= \det(\vec{a}_1\,, \: \vec{a}_2\,, \: \dotsc \,, \: \vec{a}_n) \:+\: \det(\vec{a}_1\,, \: \vec{a}_2\,, \: \dotsc \,, \: \vec{a}_n + \matrs{h\\\vdots\\0})\nnl
	&\quad \;\;+\: \det(\matrs{0\\\vdots\\h}, \: \vec{a}_2\,, \: \dotsc \,, \: \vec{a}_n) \:+\: \det(\matrs{0\\\vdots\\h}, \: \vec{a}_2\,, \: \dotsc \,, \: \matrs{h\\\vdots\\0})\,. \label{eq:detExtendedSum}
\end{align}
Since $A$ is diagonal by assumption, the column vector $\vec{a}_1$ has the form $\vec{a}_1 = \matr{a_{1}&\cdots&0}^T$ and $\vec{a}_n$ has the form $\vec{a}_n = \matr{0&\cdots&a_{n}}^T$. Therefore the vectors $\vec{a}_1$ and $\matr{h&\cdots&0}^T$ as well as $\vec{a}_n$ and $\matr{0&\cdots&h}^T$ are linearly dependent. Thus \eqref{eq:detExtendedSum} reduces to
\[
	\det(\vec{a}_1\,, \: \vec{a}_2\,, \: \dotsc \,, \: \vec{a}_n) \:+\: \det(\matrs{0\\\vdots\\h}, \: \vec{a}_2\,, \: \dotsc \,, \: \matrs{h\\\vdots\\0}) \:=\: \det A \:+\: \underbrace{\det(\matrs{0\\\vdots\\h}, \: \vec{a}_2\,, \: \dotsc \,, \: \matrs{h\\\vdots\\0})}_{:=R(h)}\,.
\]
The term $R(h)$ is quadratic in $h$, thus the linear approximation is simply%
\[
	D\det[A].H = 0\,.
\]
for a matrix $H\in\Symn$ of the form \eqref{eq:offDiagonalMatrixSpecialCase}. Through similar computations, it is easy to show that $D\det[A].H = 0$ for any off-diagonal $H\in\Symn$.

To find the derivative $D\det[A].\id$ we compute
\begin{align*}
	\det(A+h\id) &= \det\matrs{a_1+h & &0\\&\ddots &\\0&&a_n+h} = \prod_{i=1}^n (a_i+h) = \prod_{i=1}^n a_i +  \Big(\sum_{i=1}^n\prod_{\substack{j=1\\j\neq i}}^n a_j\Big)\cdot h + h^2\cdot[\dots]\,,
\end{align*} %
thus
\begin{equation}
	D\det[A].\id = \ddh \det(A+h\id) = \sum_{i=1}^n \prod_{\substack{j=1\\j\neq i}}^n a_j\,. \label{eq:detDerivativeInUnityDirection}
\end{equation}
For $n=3$ we obtain
\[
	D\det[A].\id = a_1 a_2 + a_2 a_3 + a_1 a_3\,.
\]
Furthermore, if we assume that $0$ is a simple eigenvalue of $A$ (which is the case for matrices of the form $D-\lambda_i(D)$ where $D$ is a diagonal matrix with simple eigenvalues; such matrices appeared in equation \eqref{eq:diagonalMinusSimpleEigenvalueExample}), then \eqref{eq:detDerivativeInUnityDirection} can be written as
\[
	D\det[A].\id = \prod_{\substack{j=1\\j\neq k}}^n a_j \neq 0\,,
\]
where $0$ is the $k$-th eigenvalue of $A$.

\subsection{On the derivative of isotropic functions}
\begin{lemma}
\label{lemma:isotropicFunctionDerivative}
Let $W:\Symn\to\R$ be an isotropic real valued function, i.e.
\[
	\boxed{W(Q^TXQ)=W(X) \quad \forall \: X\in\Symn\,,\:Q\in\On\,.}
\]
Then
\[
	DW[Q^TXQ] = Q^TDW[X]Q\,.
\]
\end{lemma}
\begin{proof}
We directly compute:
\begin{alignat*}{1}
	W(Q^T(X+H)Q) &= W(X+H)\\
	\Rightarrow W(Q^TXQ + Q^THQ) &= W(X) + \innerproduct{DW[X],\,H} + \dots\\
	\Rightarrow \smash{\overbrace{W(Q^TXQ)}^{=W(X)}} + \innerproduct{DW[Q^TXQ],\,Q^THQ} + \dots &= W(X) + \innerproduct{DW[X],\,H} + \dots\\
	\Rightarrow \innerproduct{DW[Q^TXQ],\,Q^THQ} &= \innerproduct{DW[X],\,H}\\
	\Rightarrow \innerproduct{QDW[Q^TXQ]Q^T,\,H} &= \innerproduct{DW[X],\,H}\,.
\end{alignat*}
Since this holds for all $H\in\Symn$, we obtain
\[
	QDW[Q^TXQ]Q^T = DW[X]
\]
and thus
\[
	DW[Q^TXQ] = Q^TDW[X]Q\,. \qedhere
\]
\end{proof}

\subsection{The eigenvalue function}
We could also try to prove Proposition \ref{prop:primaryPotential} for the more general case of non-analytic functions by directly computing the derivative of the function
\[
	W: \Symn\to\R,\quad W(A) = \sum_{i=1}^n F(\lambda_i(A))\,.
\]
Unfortunately, while the derivative of $W$ at a point $A\in\Symn$ in directions $H$ can be explicitly computed if $A$ and $H$ commute, it is difficult to do so for arbitrary choices of $H\in\Symn$.\\
One possible approach is to assume that the function $\lambda: \Symn\to\R^n$ mapping a matrix $M\in\Symn$ to its (ordered) eigenvalues $\lambda(M)$ is differentiable in a neighbourhood of $A\in\Symn$. For example, this is the case if all eigenvalues of $A$ are simple \cite{Magnus1985eigen}. The basic idea is to write $W(A) = \Psi(\lambda(A))$ with $\Psi(\lambda_1,\dotsc,\lambda_n)=\sum_{i=1}^n F(\lambda_i)$. Then
\begin{equation}
	DW[A] = D\Psi[\lambda(A)] \cdot D\lambda[A]\,. \label{eq:potentialEigenvalueDerivative}
\end{equation}
It is therefore useful to compute the derivative $D\lambda[A]$ of the eigenvalue function. Since Lemma \ref{lemma:isotropicFunctionDerivative} implies
\[
	\lambda(Q^TAQ) = \lambda(A) \:\Longrightarrow\: D\lambda[Q^TAQ] = Q^T\,D\lambda[A]\,Q\,,
\]
the derivative of $\lambda$ at $A$ is determined by the derivative at the diagonal matrix corresponding to $A$. We will therefore assume w.l.o.g. that $A$ is already a diagonal matrix.\\
The eigenvalues $\lambda_i$ of $A$ are characterized by
\begin{equation}
	\det(A-\lambda_i \id)=0\,. \label{eq:eigenvalueDefinition}
\end{equation}
Let $H\in\Symn$. We compute the first order approximation of \eqref{eq:eigenvalueDefinition}:
\begin{align}
	&\quad\:\:\,\det(A+H - \lambda_i(A+H)\cdot\id) = 0\nnl
	&\Rightarrow\: \det(A+H-[\lambda_i(A)\id + [D\lambda_i(A).H]\cdot \id + \dots]) = 0\nnl
	&\Rightarrow\: \det([A-\lambda_i(A)\id]+H-[D\lambda_i(A).H]\cdot \id + \dots) = 0\nnl
	&\Rightarrow\: \underbrace{\det[A-\lambda_i(A)\id]}_{=0} + \innerproduct{\Cof[A-\lambda_i(A)\id]\,,\: H-[D\lambda_i(A).H]\cdot \id + \dots} + \dots = 0 \label{eq:diagonalMinusSimpleEigenvalueExample}
\end{align}
By ignoring higher order terms we obtain
\begin{equation}
	\innerproduct{\Cof[A-\lambda_i(A)\id]\,,\: H-[D\lambda_i(A).H]\cdot \id} = 0 \label{eq:ignoredHigherOrderTerms}
\end{equation}
Recall that $A$ is diagonal by assumption. Since $A$ commutes with diagonal matrices $H$ (and thus the derivative $DW[A].H$ could be computed by more direct means), we are only interested in cases where the symmetric matrix $H$ is \emph{off-diagonal}, i.e.\ $H_{i,i}=0$ for $i=1,\dotsc,n$. But then
\[
	\innerproduct{\,\underbrace{\Cof[A-\lambda_i(A)\id]}_{\text{diagonal}},\: \underbrace{H\vphantom{[]}}_{\mathclap{\text{off-diagonal}}}\,} = 0\,,
\]
thus \eqref{eq:ignoredHigherOrderTerms} reduces to
\[
	\innerproduct{\Cof[A-\lambda_i(A)\id]\,,\: [D\lambda_i(A).H]\cdot \id} = 0\,,
\]
which we can also write as
\[
	(D\lambda_i(A).H)\cdot\tr\left(\Cof[A-\lambda_i(A)\id]\right) = 0\,.
\]
To conclude that $D\lambda_i(A).H=0$ it remains to show that  $\tr\left(\Cof[A-\lambda_i(A)\id]\right) \neq 0$. Assuming that the diagonal entries of $A$ are ordered we write $A=\diag(\lambda_1, \dots, \lambda_n)$ and find
\[
	A - \lambda_i(A)\id = \matr{\lambda_1-\lambda_i & & 0\\ & \ddots & \\ 0 & & \lambda_n-\lambda_i}
\]
and thus
\[
	\Cof[A - \lambda_i(A)\id] = \matr{\prod\limits_{k\neq1}(\lambda_k-\lambda_i) & & 0\\ & \ddots & \\ 0 & & \prod\limits_{k\neq n}(\lambda_k-\lambda_i)}\,.
\]
We compute the trace:
\begin{align*}
	\tr\left(\Cof[A - \lambda_i(A)\id]\right) &= \sum_{j=1}^n \, \prod\limits_{\substack{k=1\\k\neq j}}^n (\lambda_j-\lambda_i) = \prod\limits_{\substack{k=1\\k\neq i}}^n (\lambda_j-\lambda_i)\,,
\end{align*}
where the second equality holds due to the fact that the product is zero if it contains the factor $(\lambda_i-\lambda_i)$. Hence this term is nonzero if and only if all eigenvalues of $A$ are simple, in which case we can conclude that $D\lambda_i(A).H = 0$ for all off-diagonal $H\in\Symn$.

Using these results, we can prove the following, which is a simple corollary to Proposition \ref{prop:primaryPotential}:
\begin{corollary}
Let $f\in C^1(\R)$, $F\in C^2(\R)$ with $F'=f$ and let $A\in\Symn$ such that all eigenvalues of $A$ are simple. Then the function
\[
	W: \Symn\to\R,\quad W(M) = \sum_{i=1}^n F(\lambda_i(M))
\]
is differentiable at $A$ with
\[
	DW[A] = f(A) = Q^T \diag(f(\lambda_1),\dotsc,f(\lambda_n)) \,Q\,,
\]
where $A=Q^T \diag(\lambda_1,\dotsc,\lambda_n) \,Q$ is the spectral decomposition of $A$.
\end{corollary}
\begin{proof}
According to Lemma \ref{lemma:isotropicFunctionDerivative}, $DW[Q^TXQ] = Q^TDW[X]Q$, hence we find
\[
	DW[A] = Q^T DW[\diag(\lambda_1,\dotsc,\lambda_n)] \,Q\,.
\]
Therefore it remains to show that
\begin{equation}
	DW[\diag(\lambda_1,\dotsc,\lambda_n)].H = \innerproduct{\diag(f(\lambda_1),\dotsc,f(\lambda_n)),\: H} \label{eq:simpleCorollaryEquationToShow}
\end{equation}
for all $H\in\Symn$ and pairwise different $\lambda_1,\dotsc,\lambda_n$.

We first consider the case of diagonal matrices $H = \Hdiag = \diag(h_1,\dotsc,h_n)$. Writing $\Adiag = \diag(\lambda_1,\dotsc,\lambda_n)$ we find
\[
	W(\Adiag+t\Hdiag) = W(\diag(\lambda_1+th_1,\dotsc,\lambda_n+th_n)) = \sum_{i=1}^n F(\lambda_i+th_i)\,,
\]
thus
\begin{align}
	DW[\Adiag].\Hdiag &= \lim_{t\to 0} \:\tel{t} (W(\Adiag+t\Hdiag)-W(\Adiag))\\
	&= \lim_{t\to 0} \:\tel{t} \sum_{i=1}^n F(\lambda_i+th_i) - F(\lambda_i) = \sum_{i=1}^n F'(\lambda_i)\,h_i\nnl
	&= \innerproduct{\diag(f(\lambda_1),\dotsc,f(\lambda_n)),\: \diag(h_1,\dotsc,h_n)} = \innerproduct{f(\Adiag),\:\Hdiag}\,.\nonumber
\end{align}
Now let $H=\Hoff$ be a symmetric off-diagonal matrix, i.e.\ $\Hoff_{i,i}=0$ for $i=1,\dotsc,n$. Using equation \eqref{eq:potentialEigenvalueDerivative}:
\[
	DW[A] = D\Psi[\lambda(A)] \cdot D\lambda[A]\,,
\]
as well as the result of the previous considerations for diagonal $A$ and off-diagonal $\Hoff$:
\[
	D\lambda[\diag(\lambda_1,\dotsc,\lambda_n)].\Hoff = 0\,,
\]
we conclude
\[
	DW[\diag(\lambda_1,\dotsc,\lambda_n)].\Hoff = 0\,.
\]
Finally, for arbitrary $H\in\Symn$, we can write $H=\Hdiag+\Hoff$ with a diagonal matrix $\Hdiag$ and a symmetric off-diagonal matrix $\Hoff$. Then
\begin{align}
	DW[\Adiag].H &= DW[\Adiag].\Hdiag + DW[\Adiag].\Hoff = DW[\Adiag].\Hdiag\nnl
	&= \innerproduct{f(\Adiag),\:\Hdiag} = \innerproduct{f(\Adiag),\:H}\,,
\end{align}
showing \eqref{eq:simpleCorollaryEquationToShow} and concluding the proof.

\end{proof}

\end{appendix}
} %

\end{document}